\documentclass[a4paper,12pt]{amsart}

\usepackage{amsfonts}
\usepackage{amsmath}
\usepackage{amssymb}
\usepackage{graphicx}

\usepackage[usenames]{color}
\usepackage[colorlinks]{hyperref}

\newtheorem{thm}{Theorem}[section]
\newtheorem{lem}[thm]{Lemma}
\newtheorem{prop}[thm]{Proposition}

\theoremstyle{definition}

\theoremstyle{remark}
\newtheorem{rem}{Remark}[section]
\newtheorem{defn}{Definition}
\newtheorem{exam}[rem]{Example}
\numberwithin{equation}{section}

\def\d{\mathrm d}

\def\e{\mathrm e}

\begin{document}

\title[Fractional wave equation with irregular mass and dissipation]{Fractional wave equation with irregular mass and dissipation}

\author[M. Ruzhansky]{Michael Ruzhansky}
\address{
  Michael Ruzhansky:
  \endgraf
  Department of Mathematics: Analysis, Logic and Discrete Mathematics
  \endgraf
  Ghent University, Krijgslaan 281, Building S8, B 9000 Ghent
  \endgraf
  Belgium
  \endgraf
  and
  \endgraf
  School of Mathematical Sciences
  \endgraf
  Queen Mary University of London
  \endgraf
  United Kingdom
  \endgraf
  {\it E-mail address} {\rm michael.ruzhansky@ugent.be}
}

\author[M. Sebih]{Mohammed Elamine Sebih}
\address{
  Mohammed Elamine Sebih:
  \endgraf
  Laboratory of Geomatics, Ecology and Environment (LGEO2E)
  \endgraf
  University Mustapha Stambouli of Mascara, 29000 Mascara
  \endgraf
  Algeria
  \endgraf
  {\it E-mail address} {\rm sebihmed@gmail.com, ma.sebih@univ-mascara.dz}
}

\author[N. Tokmagambetov ]{Niyaz Tokmagambetov}
\address{
  Niyaz Tokmagambetov:
  \endgraf 
  Centre de Recerca Matem\'atica
  \endgraf
  Edifici C, Campus Bellaterra, 08193 Bellaterra (Barcelona), Spain
  \endgraf
  and
  \endgraf   
  Institute of Mathematics and Mathematical Modeling
  \endgraf
  125 Pushkin str., 050010 Almaty, Kazakhstan
  \endgraf  
  {\it E-mail address:} {\rm tokmagambetov@crm.cat; tokmagambetov@math.kz}
  }

\thanks{This research was funded by the Science Committee of the Ministry of Education and Science of the Republic of Kazakhstan (Grant No. AP14872042), by the FWO Odysseus 1 grant G.0H94.18N: Analysis and Partial Differential Equations, and by the Methusalem programme of the Ghent University Special Research Fund (BOF) (Grant number 01M01021). MR is also supported by EPSRC grants EP/R003025/2 and EP/V005529/1. NT is also supported by the Beatriu de Pin\'os programme and by AGAUR (Generalitat de Catalunya) grant 2021 SGR 00087.}


\begin{abstract}
In this paper, we pursue our series of papers aiming to show the applicability of the concept of very weak solutions. We consider a wave model with irregular position dependent mass and dissipation terms, in particular, allowing for $\delta$-like coefficients and prove that the problem has a very weak solution. Furthermore, we prove the uniqueness in an appropriate sense and the coherence of the very weak solution concept with classical theory. A special case of the model considered here, is the so-called telegraph equation.
\end{abstract}

\maketitle


\section{Introduction}
The telegraph equations are a system of coupled linear equations governing voltage and current flow on a linear electrical line. For $t$ denoting the time and $x$ the distance from any fixed point, and $v,\zeta$ the voltage and the current respectively, the equations are as follows
\begin{equation*}
    \bigg\lbrace
    \begin{array}{l}
    \partial_{x}\nu(t,x)=-L\partial_{t}\zeta(t,x)-R\zeta(t,x),\\
    \partial_{x}\zeta(t,x)=-C\partial_{t}\nu(t,x)-G\nu(t,x),
    \end{array}
\end{equation*}
where $L$ is the inductance, $C$ the capacitance, $R$ the resistance and $G$ stands for the conductance. When combined, a hyperbolic partial differential equation of the following form is obtained
\begin{equation}\label{Telegraph equation}
    \partial_{x}^2u(t,x) - LC\partial_{t}^2u(t,x) = (RC+GL)\partial_{t}u(t,x) + GR u(t,x),
\end{equation}
where $u$ represents either the voltage $\nu$ or the current $\zeta$. For the derivation of equations, we refer the reader to \cite{Elli23} for more details. The form \eqref{Telegraph equation} can be regarded as a wave equation with additional mass and dissipation terms. This form is widely used in the literature to study wave propagation phenomena and random walk theory. See, for instance \cite{WH93,BM98,MM93,RP95,SS09} and the references therein.

In the present paper we consider the telegraph equation in a more general case. That is, we use the fractional Laplacian instead of the classical one and for fixed $T>0$, we consider the Cauchy problem:
\begin{equation}\label{Equation intro}
    \bigg\lbrace
    \begin{array}{l}
    u_{tt}(t,x) + (-\Delta)^{s}u(t,x) + a(x)u(t,x) + b(x)u_{t}(t,x)=0, \,\, (t,x)\in [0,T]\times\mathbb{R}^d,\\
    u(0,x)=u_{0}(x),\quad u_{t}(0,x)=u_{1}(x), \quad x\in\mathbb{R}^d,
    \end{array}
\end{equation}
 where $s>0$. Motivated by the fact that mechanical and physical properties of nowadays materials can not be described by smooth functions due to the non-homogeneity of the material structure, the spatially dependent mass $a$ and the dissipation coefficient $b$ in \eqref{Equation intro} are assumed to be non-negative and singular, in particular to have $\delta$-like behaviours. Our aim is to prove that this problem is well posed in the sense of the very weak solution concept introduced in \cite{GR15} by Garetto and the first author in order to give a neat solution to the problem of multiplication that Schwartz theory of distributions is concerned with, see \cite{Sch54}, and to provide a framework in which partial differential equations involving coefficients and data of low regularity can be rigorously studied. Let us give a brief literature review about this concept of solutions. After the original work of Garetto and Ruzhansky \cite{GR15}, many researchers started using this notion of solutions for different situations, either for abstract mathematical problems as \cite{CRT21,CRT22a,CRT22b} or for physical models as in \cite{RT17a,RT17b,Gar21} and \cite{MRT19,ART19,ARST21a,ARST21b,ARST21c} where it is shown that the concept of very weak solutions is very suitable for numerical modelling, and in \cite{SW22} where the question of propagation of coefficients singularities of the very weak solution is studied. More recently, we cite \cite{GLO21,BLO22,RSY22,RY22,CDRT23}.

The novelty of this work lies in the fact that we consider equations that can not be formulated in the classical or the distributional sense. We employ the concept of very weak solutions which allows to overcome the problem of the impossibility of multiplication of distributions. Furthermore, the results obtained in this paper extend those of \cite{ARST21a}, firstly by incorporating a dissipation term, and secondly by relaxing the assumptions on the Cauchy data, allowing them to be as singular as the equation coefficients, whereas in \cite{ARST21a} they were supposed to be smooth functions.

\section{Preliminaries}
For the reader's convenience, we review in this section notations and notions that are frequently used in the sequel.
\subsection{Notation}
\begin{itemize}
    \item By the notation $f\lesssim g$, we mean that there exists a positive constant $C$, such that $f \leq Cg$ independently on $f$ and $g$.
    \item We also define
    \begin{equation*}
        \Vert u(t,\cdot)\Vert_1 := \Vert u(t,\cdot)\Vert_{L^2} + \Vert (-\Delta)^{\frac{s}{2}}u(t,\cdot)\Vert_{L^2} + \Vert u_t(t,\cdot)\Vert_{L^2},
    \end{equation*}
    and
    \begin{equation*}
        \Vert u(t,\cdot)\Vert_2 := \Vert u(t,\cdot)\Vert_{L^2} + \Vert (-\Delta)^{\frac{s}{2}}u(t,\cdot)\Vert_{L^2} + \Vert (-\Delta)^{s}u(t,\cdot)\Vert_{L^2} + \Vert u_t(t,\cdot)\Vert_{L^2}.
    \end{equation*}
\end{itemize}
We also recall the well-known H\"older inequality.
\begin{prop}\label{Holder inequality}
    Let $r\in(0,\infty)$ and $p,q\in (0,\infty)$ be such that $\frac{1}{r}=\frac{1}{p} + \frac{1}{q}$. Assume that $f\in L^{p}(\mathbb{R}^d)$ and $g\in L^{q}(\mathbb{R}^d)$, then, $fg\in L^{r}(\mathbb{R}^d)$ and we have
    \begin{equation}
        \Vert fg\Vert_{L^r}\leq \Vert f\Vert_{L^p}\Vert g\Vert_{L^q}.
    \end{equation}
\end{prop}

\subsection{The fractional Sobolev space $H^s$ and the fractional Laplacian}
\begin{defn}[\textbf{Fractional Sobolev space}]\label{Def. frac. sobolev space}
    Given $s>0$, the fractional Sobolev space is defined by
    \begin{equation*}
        H^s(\mathbb{R}^d) = \big\{f\in L^2(\mathbb{R}^d) : \int_{\mathbb{R}^d}(1+\vert \xi\vert^{2s})|\widehat{f}(\xi)|^2\d \xi < +\infty\big\},
    \end{equation*}
    where $\widehat{f}$ denotes the Fourier transform of $f$.
\end{defn}
We note that, the fractional Sobolev space $H^s$ endowed with the norm
    \begin{equation}\label{Norm H^s}
        \Vert f\Vert_{H^s}:=\bigg(\int_{\mathbb{R}^d}(1+\vert \xi\vert^{2s})|\widehat{f}(\xi)|^2\d \xi\bigg)^{\frac{1}{2}},\quad\text{for } f\in H^s(\mathbb{R}^d),
    \end{equation}
    is a Hilbert space.
\begin{defn}[\textbf{Fractional Laplacian}]\label{Def. fract. laplacian}
    For $s>0$, $(-\Delta)^s$ denotes the fractional Laplacian defined by
    \begin{equation*}
    (-\Delta)^{s}f = \mathcal{F}^{-1}(\vert\xi\vert^{2s}(\widehat{f})),
    \end{equation*}
    for all $\xi\in \mathbb{R}^d$.
\end{defn}
In other words, the fractional Laplacian $(-\Delta)^s$ can be viewed as the pseudo-differential operator with symbol $\vert \xi\vert^{2s}$. With this definition and the Plancherel theorem, the fractional Sobolev space can be defined as:
    \begin{equation}
        H^{s}(\mathbb{R}^{d})=\big\{ f\in L^{2}(\mathbb{R}^{d}): (-\Delta)^{\frac{s}{2}}f \in L^{2}(\mathbb{R}^{d})\big\},
    \end{equation}
    moreover, the norm
    \begin{equation}
        \Vert f\Vert_{H^{s}}:=\Vert f\Vert_{L^2}+\Vert (-\Delta)^{\frac{s}{2}}f\Vert_{L^2},
    \end{equation}
    is equivalent to the one defined in \eqref{Norm H^s}.

\begin{rem}
    We note that the fractional Sobolev space $H^s(\mathbb{R}^d)$ can also be defined via the Gagliardo norm, however, we chose this approach, since it is valid for any real $s>0$, unlike the one via Gagliardo norm which is valid only for $s\in (0,1)$. We refer the reader to \cite{DPV12,Gar18,Kwa17} for more details and alternative definitions.
\end{rem}

\begin{prop}[\textbf{Fractional Sobolev inequality}, e.g. Theorem 1.1. \cite{CT04}]\label{Prop. Sobolev estimate}
    For $d\in \mathbb{N}_0$ and $s\in \mathbb{R}_+$, let $d>2s$ and $q=\frac{2d}{d-2s}$. Then, the estimate
    \begin{equation}\label{Sobolev estimate}
        \Vert f\Vert_{L^q} \leq C(d,s)\Vert (-\Delta)^{\frac{s}{2}}f\Vert_{L^2},
    \end{equation}
    holds for all $f\in H^{s}(\mathbb{R}^d)$, where the constant $C$ depends only on the dimension $d$ and the order $s$.
\end{prop}

\subsection{Duhamel's principle}
We prove the following special version of Duhamel's principle that will frequently be used throughout this paper. For more general versions of this principle,  we refer the reader to \cite{ER18}. Let us consider the following Cauchy problem,
\begin{equation}\label{Equation Duhamel}
    \left\lbrace
    \begin{array}{l}
    u_{tt}(t,x)+\lambda(x)u_{t}(t,x)+Lu(t,x)=f(t,x) ,~~~(t,x)\in\left(0,\infty\right)\times \mathbb{R}^{d},\\
    u(0,x)=u_{0}(x),\,\,\, u_{t}(0,x)=u_{1}(x),\,\,\, x\in\mathbb{R}^{d},
    \end{array}
    \right.
\end{equation}
for a given function $\lambda$ and $L$ is a linear partial differential operator acting over the spatial variable.

\begin{prop}\label{Prop Duhamel}
The solution to the Cauchy problem (\ref{Equation Duhamel}) is given by
\begin{equation}\label{Sol Duhamel}
    u(t,x)= w(t,x) + \int_0^t v(t,x;\tau)\d \tau,
\end{equation}
where $w(t,x)$ is the solution to the homogeneous problem
\begin{equation}\label{Homog eqn Duhamel}
    \left\lbrace
    \begin{array}{l}
    w_{tt}(t,x)+\lambda(x)w_{t}(t,x)+Lw(t,x)=0 ,~~~(t,x)\in\left(0,\infty\right)\times \mathbb{R}^{d},\\
    w(0,x)=u_{0}(x),\,\,\, w_{t}(0,x)=u_{1}(x),\,\,\, x\in\mathbb{R}^{d},
    \end{array}
    \right.
\end{equation}
and $v(t,x;\tau)$ solves the auxiliary Cauchy problem
\begin{equation}\label{Aux eqn Duhamel}
    \left\lbrace
    \begin{array}{l}
    v_{tt}(t,x;\tau)+\lambda(x)v_{t}(t,x;\tau)+Lv(t,x;\tau)=0 ,~~~(t,x)\in\left(\tau,\infty\right)\times \mathbb{R}^{d},\\
    v(\tau,x;\tau)=0,\,\,\, v_{t}(\tau,x;\tau)=f(\tau,x),\,\,\, x\in\mathbb{R}^{d},
    \end{array}
    \right.
\end{equation}
where $\tau$ is a parameter varying over $\left(0,\infty\right)$.
\end{prop}

\begin{proof}
    Firstly, we apply $\partial_{t}$ to $u$ in \eqref{Sol Duhamel}. We get
    \begin{equation}\label{sol Duhamel u_t}
        \partial_t u(t,x)=\partial_t w(t,x) + \int_0^t \partial_t v(t,x;\tau)\d \tau,
    \end{equation}
    and accordingly
    \begin{equation}\label{sol Duhamel lambdau_t}
        \lambda(x)\partial_t u(t,x)=\lambda(x)\partial_t w(t,x) + \int_0^t \lambda(x)\partial_t v(t,x;\tau)\d \tau,
    \end{equation}
    where we used the fact that $v(t,x;t)=0$ by the imposed initial condition in \eqref{Aux eqn Duhamel}. We differentiate again \eqref{sol Duhamel u_t} with respect to $t$ to get
    \begin{equation}\label{sol Duhamel u_tt}
        \partial_{tt} u(t,x)=\partial_{tt} w(t,x) + f(t,x) + \int_0^t \partial_{tt} v(t,x;\tau)\d \tau,
    \end{equation}
    where we used that $\partial_t v(t,x;t) = f(t,x)$. Now, applying $L$ to $u$ in \eqref{Sol Duhamel} gives
    \begin{equation}\label{sol Duhamel L}
        L u(t,x)=L w(t,x) + \int_0^t L v(t,x;\tau)\d \tau.
    \end{equation}
    By adding \eqref{sol Duhamel u_tt}, \eqref{sol Duhamel L} and \eqref{sol Duhamel lambdau_t} with in mind that $w$ and $v$ satisfy the equations in \eqref{Homog eqn Duhamel} and \eqref{Aux eqn Duhamel}, we get
    \begin{equation*}
        u_{tt}(t,x)+\lambda(x)u_{t}(t,x)+Lu(t,x)=f(t,x).
    \end{equation*}
    It remains to prove that $u$ satisfy the initial conditions. Indeed, from \eqref{Sol Duhamel} and \eqref{sol Duhamel u_t}, we have that $u(0,x)=w(0,x)=u_0(x)$ and that $u_t(0,x)=\partial_t w(0,x)=u_1(x)$. This concludes the proof.    
\end{proof}

\begin{rem}
    We note that the above statement of Duhamel's principle can be extended to differential operators of order $k\in \mathbb{N}$. Indeed, if we consider the Cauchy problem
    \begin{equation*}\label{Equation Duhamel general}
        \left\lbrace
        \begin{array}{l}        \partial_{t}^{k}u(t,x)+\sum_{j=1}^{k-1}\lambda_{j}(x)\partial_{t}^{j} u(t,x)+Lu(t,x)=f(t,x) ,~~~(t,x)\in\left(0,\infty\right)\times \mathbb{R}^{d},\\
        \partial_{t}^{j}u(0,x)=u_{j}(x),\,\text{for}\quad j=0,\cdots,k-1, \quad x\in\mathbb{R}^{d},
        \end{array}
        \right.
    \end{equation*}
    then, the solution is given by
\begin{equation*}\label{Sol Duhamel general}
    u(t,x)= w(t,x) + \int_0^t v(t,\tau;\tau)\d \tau,
\end{equation*}
where $w(t,x)$ is the solution to the homogeneous problem
\begin{equation*}\label{Homog eqn Duhamel general}
    \left\lbrace
    \begin{array}{l}
    \partial_{t}^{k}w(t,x)+\sum_{j=1}^{k-1}\lambda_{j}(x)\partial_{t}^{j}w(t,x)+Lw(t,x)=0 ,~~~(t,x)\in\left(0,\infty\right)\times \mathbb{R}^{d},\\
    \partial_{t}^{j}w(0,x)=u_{j}(x),\,\text{for}\quad j=0,\cdots,k-1, \quad x\in\mathbb{R}^{d},
    \end{array}
    \right.
\end{equation*}
and $v(t,x;\tau)$ solves the auxiliary Cauchy problem
\begin{equation*}\label{Aux eqn Duhamel general}
    \left\lbrace
    \begin{array}{l}
    \partial_{t}^{k}v(t,x;\tau)+\sum_{j=1}^{k-1}\lambda_{j}(x)\partial_{t}^{j}v_{t}(t,x;\tau)+Lv(t,x;\tau)=0 ,~~~(t,x)\in\left(\tau,\infty\right)\times \mathbb{R}^{d},\\
    \partial_{t}^{j}w(\tau,x;\tau)=0,\,\text{for}\quad j=0,\cdots,k-2,,\,\,\, \partial_{t}^{k-1}w(\tau,x;\tau)=f(\tau,x),\,\,\, x\in\mathbb{R}^{d},
    \end{array}
    \right.
\end{equation*}
where $\tau\in \left(0,\infty\right)$.
\end{rem}

\subsection{Energy estimates for the classical solution}
In order to prove existence and uniqueness of a very weak solution to the Cauchy problem \eqref{Equation intro} as well as the coherence with classical theory we will often use the following lemmas that are stated in the case when the mass $a$ and the dissipation coefficient $b$ are regular functions. The statements of the lemmas are given under different assumptions on $a$ and $b$.
\begin{lem}\label{Lemma1}
    Let $a,b\in L^{\infty}(\mathbb{R}^d)$ be non-negative and suppose that $u_0 \in H^{s}(\mathbb{R}^d)$ and $u_1 \in L^{2}(\mathbb{R}^d)$. Then the unique solution $u\in C([0,T];H^{s}(\mathbb{R}^d))\cap C^1([0,T];L^{2}(\mathbb{R}^d))$ to the Cauchy problem \eqref{Equation intro} satisfies the estimate
    \begin{equation}\label{Energy estimate}
         \Vert u(t,\cdot)\Vert_1 \lesssim \Big(1+\Vert a\Vert_{L^{\infty}}\Big)\Big(1+\Vert b\Vert_{L^{\infty}}\Big)\bigg[\Vert u_0\Vert_{H^s} + \Vert u_{1}\Vert_{L^2}\bigg],
    \end{equation}
     for all $t\in [0,T]$.
\end{lem}

\begin{proof}
    Multiplying the equation in \eqref{Equation intro} by $u_t$ and integrating with respect to the variable $x$ over $\mathbb{R}^d$ and taking the real part, we get
    \begin{align} \label{Energy functional}
        Re\Bigg(\langle u_{tt}(t,\cdot),&u_{t}(t,\cdot)\rangle_{L^2} + \langle (-\Delta)^{s}u(t,\cdot),u_{t}(t,\cdot)\rangle_{L^2}\\
        & + \langle a(\cdot)u(t,\cdot),u_{t}(t,\cdot)\rangle_{L^2} + \langle b(\cdot)u_{t}(t,\cdot),u_{t}(t,\cdot)\rangle_{L^2}\Bigg) = 0.\nonumber
    \end{align}
    We easily see that
    \begin{equation}\label{En1}
        Re\langle u_{tt}(t,\cdot),u_{t}(t,\cdot)\rangle_{L^2} = \frac{1}{2}\partial_{t}\langle u_{t}(t,\cdot),u_{t}(t,\cdot)\rangle_{L^2} = \frac{1}{2}\partial_{t}\Vert u_{t}(t,\cdot)\Vert_{L^2}^2,
    \end{equation}
    and
    \begin{align}\label{En2}
        Re\langle (-\Delta)^{s}u(t,\cdot),u_{t}(t,\cdot)\rangle_{L^2} &= \frac{1}{2}\partial_{t}\langle (-\Delta)^{\frac{s}{2}}u(t,\cdot),(-\Delta)^{\frac{s}{2}}u(t,\cdot)\rangle_{L^2} \\
        & = \frac{1}{2}\partial_{t}\Vert (-\Delta)^{\frac{s}{2}}u(t,\cdot)\Vert_{L^2}^2,
    \end{align}
    where we used the self-adjointness of the operator $(-\Delta)^s$. For the remaining terms in \eqref{Energy functional}, we have
    \begin{equation}\label{En3}
        Re\langle a(\cdot)u(t,\cdot),u_{t}(t,\cdot)\rangle_{L^2} = \frac{1}{2}\partial_t\Vert a^{\frac{1}{2}}(\cdot)u(t,\cdot)\Vert_{L^2}^2,
    \end{equation}
    and
    \begin{equation}\label{En4}
        Re\langle b(\cdot)u_{t}(t,\cdot),u_{t}(t,\cdot)\rangle_{L^2} = \Vert b^{\frac{1}{2}}(\cdot)u_{t}(t,\cdot)\Vert_{L^2}^2.
    \end{equation}
    By substituting \eqref{En1},\eqref{En2},\eqref{En3} and \eqref{En4} in \eqref{Energy functional} we get
    \begin{equation}\label{Energy functional1}
        \partial_{t}\Big[ \Vert u_{t}(t,\cdot)\Vert_{L^2}^2 + \Vert (-\Delta)^{\frac{s}{2}}u(t,\cdot)\Vert_{L^2}^2 + \Vert a^{\frac{1}{2}}(\cdot)u(t,\cdot)\Vert_{L^2}^2\Big] = -2\Vert b^{\frac{1}{2}}(\cdot)u_{t}(t,\cdot)\Vert_{L^2}^2.
    \end{equation}
    Let us denote
    \begin{equation}
        E(t):= \Vert u_{t}(t,\cdot)\Vert_{L^2}^2 + \Vert (-\Delta)^{\frac{s}{2}}u(t,\cdot)\Vert_{L^2}^2 + \Vert a^{\frac{1}{2}}(\cdot)u(t,\cdot)\Vert_{L^2}^2,
    \end{equation}
    the energy function of the system \eqref{Equation intro}. It follows from \eqref{Energy functional1} that $\partial_{t}E(t)\leq 0$ and consequently that we have a decay of energy, that is: $E(t)\leq E(0)$ for all $t\in [0,T]$. By taking into consideration the estimate
    \begin{equation}
        \Vert a^{\frac{1}{2}}(\cdot)u_{0}\Vert_{L^2}^2 \leq \Vert a\Vert_{L^{\infty}}\Vert u_0\Vert_{L^2}^2,
    \end{equation}
    it follows that all terms in $E(t)$ satisfy the estimates:
    \begin{align}\label{Energy1}
        \Vert a^{\frac{1}{2}}(\cdot)u(t,\cdot)\Vert_{L^2}^2 & \lesssim \Vert u_{1}\Vert_{L^2}^2 + \Vert (-\Delta)^{\frac{s}{2}}u_0\Vert_{L^2}^2 + \Vert a\Vert_{L^{\infty}}\Vert u_0\Vert_{L^2}^2\\
        & \lesssim \Vert u_{1}\Vert_{L^2}^2 + \Vert u_0\Vert_{H^s}^2 + \Vert a\Vert_{L^{\infty}}\Vert u_0\Vert_{H^s}^2\nonumber\\
        & \lesssim \big(1+\Vert a\Vert_{L^{\infty}}\big)\big[\Vert u_0\Vert_{H^s}^2 + \Vert u_{1}\Vert_{L^2}^2\big]\nonumber\\
        & \lesssim \big(1+\Vert a\Vert_{L^{\infty}}\big)\big[\Vert u_0\Vert_{H^s} + \Vert u_{1}\Vert_{L^2}\big]^2,\nonumber
    \end{align}
    as well as
    \begin{equation}\label{Energy2}
        \bigg\{\Vert u_{t}(t,\cdot)\Vert_{L^2}^2, \Vert (-\Delta)^{\frac{s}{2}}u(t,\cdot)\Vert_{L^2}^2\bigg\} \lesssim \big(1+\Vert a\Vert_{L^{\infty}}\big)\big[\Vert u_0\Vert_{H^s} + \Vert u_{1}\Vert_{L^2}\big]^2,
    \end{equation}
    uniformly in $t\in [0,T]$, where we used the fact that: $\lambda^2 + \gamma^2 \leq (\lambda+\gamma)^2$ for all $\lambda,\gamma\in \mathbb{R}_+$ and
    \begin{equation*}
        \Big\{ \Vert (-\Delta)^{\frac{s}{2}}u_0\Vert_{L^2}, \Vert u_0\Vert_{L^2}\Big\} \leq \Vert u_0\Vert_{H^s}.
    \end{equation*}
    We now need to estimate $u$. For this purpose, we apply the Fourier transform to \eqref{Equation intro} with respect to the variable $x$ to get the non-homogeneous ordinary differential equation
    \begin{equation}\label{Equation Fourier transf}
        \widehat{u}_{tt}(t,\xi)+\vert \xi\vert^{2s}\widehat{u}(t,\xi)=\widehat{f}(t,\xi), \quad (t,\xi)\in [0,T]\times\mathbb{R}^d,
    \end{equation}
    with the initial conditions $\widehat{u}(0,\xi)=\widehat{u}_{0}(\xi)$ and $\widehat{u}_{t}(0,\xi)=\widehat{u}_{1}(\xi)$. Here $\widehat{f}$, $\widehat{u}$ denote the Fourier transform of $f$ and $u$ respectively, where $f(t,x):=-a(x)u(t,x)-b(x)u_{t}(t,x)$. Treating $\widehat{f}(t,\xi)$ as a source term and using Duhamel's principle (Proposition \ref{Prop Duhamel} with $\lambda \equiv0$) to solve \eqref{Equation Fourier transf}, we derive the following representation of the solution,
    \begin{equation}\label{Representation of Fourier sol}
        \widehat{u}(t,\xi)=\cos (t\vert\xi\vert^{s}) \widehat{u}_{0}(\xi) + \frac{\sin (t\vert\xi\vert^{s})}{\vert\xi\vert^{s}} \widehat{u}_{1}(\xi) + \int_{0}^{t}\frac{\sin ((t-\tau)\vert\xi\vert^{s})}{\vert\xi\vert^{s}} \widehat{f}(\tau,\xi) \d\tau.
    \end{equation}
    Taking the $L^2$ norm in \eqref{Representation of Fourier sol} and using the estimates:
    \begin{enumerate}
        \item[1.] $\vert \cos (t\vert\xi\vert^{s})\vert \leq 1$, for $t\in \left[0,T\right]$ and $\xi\in \mathbb{R}^{d}$,
        \item[2.] $\vert \sin (t\vert\xi\vert^{s})\vert \leq 1$, for large frequencies and $t\in \left[0,T\right]$ and
        \item[3.] $\vert \sin (t\vert\xi\vert^{s})\vert \leq t\vert\xi\vert^{s} \leq T\vert\xi\vert^{s}$, for small frequencies and $t\in \left[0,T\right]$,
    \end{enumerate}
    leads to
    \begin{equation*}
        \Vert \widehat{u}(t,\cdot)\Vert_{L^2}^{2} \lesssim \Vert \widehat{u}_{0}\Vert_{L^2}^{2} + \Vert \widehat{u}_{1}\Vert_{L^2}^{2} + \int_{0}^{t}\Vert \widehat{f}(\tau,\cdot)\Vert_{L^2}^{2}\d \tau,
    \end{equation*}
    and by using the Parseval-Plancherel identity, we get
    \begin{equation}\label{Estimate1 u}
        \Vert u(t,\cdot)\Vert_{L^2}^{2} \lesssim \Vert u_{0}\Vert_{L^2}^{2} + \Vert u_{1}\Vert_{L^2}^{2} + \int_{0}^{t}\Vert f(\tau,\cdot)\Vert_{L^2}^{2}\d \tau,
    \end{equation}
    for all $t\in [0,T]$.
    To estimate $\Vert f(\tau,\cdot)\Vert_{L^2}$, the last term in the above inequality, we use the triangle inequality and the estimates
    \begin{align}
        \Vert a(\cdot)u(\tau,\cdot)\Vert_{L^2} & \leq \Vert a\Vert_{L^{\infty}}^{\frac{1}{2}} \Vert a^{\frac{1}{2}}(\cdot)u(\tau,\cdot)\Vert_{L^2}\\
        & \lesssim \Vert a\Vert_{L^{\infty}}^{\frac{1}{2}} \big(1+\Vert a\Vert_{L^{\infty}}\big)^{\frac{1}{2}}\big[\Vert u_0\Vert_{H^s} + \Vert u_{1}\Vert_{L^2}\big] \nonumber\\
        & \lesssim \big(1+\Vert a\Vert_{L^{\infty}}\big)\big[\Vert u_0\Vert_{H^s} + \Vert u_{1}\Vert_{L^2}\big], \nonumber
    \end{align}
    resulting from \eqref{Energy1}, and similarly
    \begin{align}
        \Vert b(\cdot)u_{t}(\tau,\cdot)\Vert_{L^2} & \leq \Vert b\Vert_{L^{\infty}} \Vert u_{t}(\tau,\cdot)\Vert_{L^2}\\
        & \lesssim \Vert b\Vert_{L^{\infty}}\big(1+\Vert a\Vert_{L^{\infty}}\big)\big[\Vert u_0\Vert_{H^s} + \Vert u_{1}\Vert_{L^2}\big], \nonumber
    \end{align}
    resulting from \eqref{Energy2}, to get
    \begin{equation}\label{Estimate f}
        \Vert f(\tau,\cdot)\Vert_{L^2} \lesssim \big(1+\Vert a\Vert_{L^{\infty}}\big)\big(1+\Vert b\Vert_{L^{\infty}}\big)\big[\Vert u_0\Vert_{H^s} + \Vert u_{1}\Vert_{L^2}\big].
    \end{equation}
    The desired estimate for $u$ follows by substituting \eqref{Estimate f} into \eqref{Estimate1 u}, finishing the proof.        
\end{proof}

\begin{lem}\label{Lemma2}
    Let $d>2s$. Assume that $a\in L^{\frac{d}{s}}(\mathbb{R}^d) \cap L^{\frac{d}{2s}}(\mathbb{R}^d)$ and $b\in L^{\frac{d}{s}}(\mathbb{R}^d)$ be non-negative. If $u_0\in H^{2s}(\mathbb{R}^d)$ and $u_1\in H^{s}(\mathbb{R}^d)$, then, there is a unique solution $u\in C([0,T]; H^{2s}(\mathbb{R}^d))\cap C^{1}([0,T]; H^{s}(\mathbb{R}^d))$ to \eqref{Equation intro} and it satisfies the estimate
    \begin{equation}\label{Energy estimate1}
         \Vert u(t,\cdot)\Vert_2 \lesssim \Big(1+\Vert a\Vert_{L^{\frac{d}{s}}}\Big)\Big(1+\Vert a\Vert_{L^{\frac{d}{2s}}}\Big)\Big(1+\Vert b\Vert_{L^{\frac{d}{s}}}\Big)^2\bigg[\Vert u_0\Vert_{H^{2s}} + \Vert u_{1}\Vert_{H^s}\bigg],
    \end{equation}
     uniformly in $t\in [0,T]$.
\end{lem}
\begin{proof}
    Proceeding as in the proof of Lemma \ref{Lemma1}, we get
    \begin{equation}\label{Energy functional1.1}
        \partial_{t}E(t) = -2\Vert b^{\frac{1}{2}}(\cdot)u_{t}(t,\cdot)\Vert_{L^2}^2\leq 0,
    \end{equation}
    for the energy function of the system defined by
    \begin{equation}\label{Energy functional1.2}
        E(t)= \Vert u_{t}(t,\cdot)\Vert_{L^2}^2 + \Vert (-\Delta)^{\frac{s}{2}}u(t,\cdot)\Vert_{L^2}^2 + \Vert a^{\frac{1}{2}}(\cdot)u(t,\cdot)\Vert_{L^2}^2,
    \end{equation}
    which implies the decay of the energy over $t$. That is
    \begin{equation}\label{Energy decay}
        E(t)\leq \Vert u_{1}\Vert_{L^2}^2 + \Vert (-\Delta)^{\frac{s}{2}}u_{0}\Vert_{L^2}^2 + \Vert a^{\frac{1}{2}}(\cdot)u_{0}\Vert_{L^2}^2,
    \end{equation}
    for all $t\in[0,T]$. Using H\"older's inequality (see Proposition \ref{Holder inequality}) for the last term in \eqref{Energy decay} together with $\Vert a^{\frac{1}{2}}\Vert_{L^p}^2 = \Vert a\Vert_{L^{\frac{p}{2}}}$, gives
    \begin{equation}
        \Vert a^{\frac{1}{2}}(\cdot)u_{0}(\cdot)\Vert_{L^2}^2 \leq \Vert a\Vert_{L^{\frac{p}{2}}}\Vert u_0\Vert_{L^q}^2,
    \end{equation}
    for $1<p,q<\infty$, satisfying $\frac{1}{p}+\frac{1}{q}=\frac{1}{2}$. Now, if we choose $q=\frac{2d}{d-2s}$ and consequently $p=\frac{d}{s}$, it follows from Proposition \ref{Prop. Sobolev estimate} that
    \begin{equation}
        \Vert u_0\Vert_{L^q} \lesssim \Vert (-\Delta)^{\frac{s}{2}}u_{0}(\cdot)\Vert_{L^2}\leq \Vert u_{0}\Vert_{H^s},
    \end{equation}
    and thus
    \begin{equation}\label{Estimate au_0}
        \Vert a^{\frac{1}{2}}(\cdot)u_{0}(\cdot)\Vert_{L^2}^2 \lesssim \Vert a\Vert_{L^{\frac{d}{2s}}}\Vert u_{0}\Vert_{H^s}^2.
    \end{equation}
    Substituting \eqref{Estimate au_0} in \eqref{Energy decay}, we get the estimates
    \begin{equation}\label{Energy3}
        \bigg\{\Vert u_{t}(t,\cdot)\Vert_{L^2}^2, \Vert (-\Delta)^{\frac{s}{2}}u(t,\cdot)\Vert_{L^2}^2, \Vert a^{\frac{1}{2}}(\cdot)u(t,\cdot)\Vert_{L^2}^2\bigg\} \lesssim \big(1+\Vert a\Vert_{L^{\frac{d}{2s}}}\big)\big[\Vert u_0\Vert_{H^s} + \Vert u_{1}\Vert_{L^2}\big]^2,
    \end{equation}
    uniformly in $t\in [0,T]$. To prove the estimate for the solution $u$, we argue as in the proof of Lemma \ref{Lemma1} to get
    \begin{equation}\label{Estimate2 u}
        \Vert u(t,\cdot)\Vert_{L^2}^{2} \lesssim \Vert u_{0}\Vert_{L^2}^{2} + \Vert u_{1}\Vert_{L^2}^{2} + \int_{0}^{t}\Vert f(\tau,\cdot)\Vert_{L^2}^{2}\d\tau,
    \end{equation}
    for all $t\in [0,T]$, with $f(t,x):=-a(x)u(t,x)-b(x)u_{t}(t,x)$. In order to estimate $\Vert f(\tau,\cdot)\Vert_{L^2}$, we use the triangle inequality to get
    \begin{equation}\label{Estimate2 f}
        \Vert f(t,\cdot)\Vert_{L^2}\leq \Vert a(\cdot)u(t,\cdot)\Vert_{L^2} + \Vert b(\cdot)u_{t}(t,\cdot)\Vert_{L^2}.
    \end{equation}
    To estimate the first term in \eqref{Estimate2 f}, we first use H\"older's inequality together with $\Vert a^{2}\Vert_{L^{\frac{p}{2}}} = \Vert a\Vert_{L^p}^2$, to get
    \begin{equation}
        \Vert a(\cdot)u(t,\cdot)\Vert_{L^2} \leq \Vert a\Vert_{L^p} \Vert u(t,\cdot)\Vert_{L^q},
    \end{equation}
    for $1<p,q<\infty$, satisfying $\frac{1}{p}+\frac{1}{q}=\frac{1}{2}$, and we choose $q=\frac{2d}{d-2s}$ and consequently $p=\frac{d}{s}$, in order to get (from Proposition \ref{Prop. Sobolev estimate})
    \begin{equation}
        \Vert u(t,\cdot)\Vert_{L^q} \lesssim \Vert (-\Delta)^{\frac{s}{2}}u(t,\cdot)\Vert_{L^2},
    \end{equation}
    and thus
    \begin{equation}\label{Estimate au}
        \Vert a(\cdot)u(t,\cdot)\Vert_{L^2} \lesssim \Vert a\Vert_{L^{\frac{d}{s}}} \Vert (-\Delta)^{\frac{s}{2}}u(t,\cdot)\Vert_{L^2},
    \end{equation}
    for all $t\in[0,T]$. Using the estimate \eqref{Energy3}, we arrive at
    \begin{align}\label{Estimate2 au}
        \Vert a(\cdot)u(t,\cdot)\Vert_{L^2} & \lesssim \Vert a\Vert_{L^{\frac{d}{s}}} \big(1+\Vert a\Vert_{L^{\frac{d}{2s}}}\big)^{\frac{1}{2}}\big[\Vert u_0\Vert_{H^s} + \Vert u_{1}\Vert_{L^2}\big]\\
        & \nonumber \lesssim \Vert a\Vert_{L^{\frac{d}{s}}} \big(1+\Vert a\Vert_{L^{\frac{d}{2s}}}\big)\big[\Vert u_0\Vert_{H^s} + \Vert u_{1}\Vert_{L^2}\big],
    \end{align}
    where we used $\lambda^{\frac{1}{2}}\leq \lambda$ for $\lambda \geq 1$. For the second term in \eqref{Estimate2 f}, we argue as above, to get
    \begin{equation}\label{Estimate bu_t}
        \Vert b(\cdot)u_t(t,\cdot)\Vert_{L^2} \lesssim \Vert b\Vert_{L^{\frac{d}{s}}} \Vert (-\Delta)^{\frac{s}{2}}u_t(t,\cdot)\Vert_{L^2},
    \end{equation}
    for all $t\in[0,T]$. We need now to estimate $\Vert (-\Delta)^{\frac{s}{2}}u_t(t,\cdot)\Vert_{L^2}$. For this, we note that if $u$ solves the Cauchy problem
    \begin{equation}
        \bigg\lbrace
        \begin{array}{l}
        u_{tt}(t,x) + (-\Delta)^{s}u(t,x) + a(x)u(t,x) + b(x)u_{t}(t,x)=0, \,\, (t,x)\in [0,T]\times\mathbb{R}^d,\\
        u(0,x)=u_{0}(x),\quad u_{t}(0,x)=u_{1}(x), \quad x\in\mathbb{R}^d,
        \end{array}
    \end{equation}
    then $u_t$ solves
    \begin{equation}\label{Equation in u_t}
        \bigg\lbrace
        \begin{array}{l}
        (u_t)_{tt}(t,x) + (-\Delta)^{s}u_t(t,x) + a(x)u_t(t,x) + b(x)(u_t)_{t}(t,x)=0, \,\, (t,x)\in [0,T]\times\mathbb{R}^d,\\
        u_t(0,x)=u_{1}(x),\quad u_{tt}(0,x)=-(-\Delta)^{s}u_0(x) - a(x)u_0(x) - b(x)u_1(x), \quad x\in\mathbb{R}^d.
        \end{array}
    \end{equation}
    Thanks to \eqref{Estimate au} and \eqref{Estimate bu_t} one has
    \begin{equation}
        \Vert a(\cdot)u_0(\cdot)\Vert_{L^2} \lesssim \Vert a\Vert_{L^{\frac{d}{s}}} \Vert u_0\Vert_{H^s},\quad \Vert b(\cdot)u_1(\cdot)\Vert_{L^2} \lesssim \Vert b\Vert_{L^{\frac{d}{s}}} \Vert u_1\Vert_{H^s}.
    \end{equation}
    The estimate for $\Vert (-\Delta)^{\frac{s}{2}}u_t(t,\cdot)\Vert_{L^2}$ follows by using \eqref{Energy3} applied to the problem \eqref{Equation in u_t}, to get
    \begin{align}\label{Estimate Delta u_t}
        \Vert (-\Delta)^{\frac{s}{2}}u_t(t,\cdot)\Vert_{L^2} & \lesssim \big(1+\Vert a\Vert_{L^{\frac{d}{2s}}}\big)^{\frac{1}{2}}\big[\Vert u_1\Vert_{H^s} + \Vert u_{tt}(0,\cdot)\Vert_{L^2}\big]\\
        & \nonumber \lesssim \big(1+\Vert a\Vert_{L^{\frac{d}{2s}}}\big)^{\frac{1}{2}}\big[\Vert u_1\Vert_{H^s} + \Vert u_0\Vert_{H^{2s}} + \Vert a\Vert_{L^{\frac{d}{s}}} \Vert u_0\Vert_{H^s} + \Vert b\Vert_{L^{\frac{d}{s}}} \Vert u_1\Vert_{H^s}\big]\\
        & \nonumber \lesssim \big(1+\Vert a\Vert_{L^{\frac{d}{2s}}}\big)\big(1+\Vert a\Vert_{L^{\frac{d}{s}}}\big)\big(1+\Vert b\Vert_{L^{\frac{d}{s}}}\big)\big[\Vert u_0\Vert_{H^{2s}} + \Vert u_1\Vert_{H^s}\big].
    \end{align}
    By substituting \eqref{Estimate Delta u_t} in \eqref{Estimate bu_t} we get
    \begin{align}\label{Estimate2 bu_t}
        \Vert b(\cdot)u_t(t,\cdot)\Vert_{L^2} & \lesssim \Vert b\Vert_{L^{\frac{d}{s}}}\big(1+\Vert a\Vert_{L^{\frac{d}{2s}}}\big)\big(1+\Vert a\Vert_{L^{\frac{d}{s}}}\big)\big(1+\Vert b\Vert_{L^{\frac{d}{s}}}\big)\big[\Vert u_0\Vert_{H^{2s}} + \Vert u_1\Vert_{H^s}\big]\nonumber\\
        & \lesssim \big(1+\Vert a\Vert_{L^{\frac{d}{2s}}}\big)\big(1+\Vert a\Vert_{L^{\frac{d}{s}}}\big)\big(1+\Vert b\Vert_{L^{\frac{d}{s}}}\big)^2\big[\Vert u_0\Vert_{H^{2s}} + \Vert u_1\Vert_{H^s}\big],
    \end{align}
    and the estimate for $\Vert f(t,\cdot)\Vert_{L^2}$ follows from  \eqref{Estimate f} and \eqref{Estimate2 au} with \eqref{Estimate2 bu_t}, yielding
    \begin{equation}\label{Estimate3 f}
        \Vert f(t,\cdot)\Vert_{L^2} \lesssim \big(1+\Vert a\Vert_{L^{\frac{d}{2s}}}\big)\big(1+\Vert a\Vert_{L^{\frac{d}{s}}}\big)\big(1+\Vert b\Vert_{L^{\frac{d}{s}}}\big)^2\big[\Vert u_0\Vert_{H^{2s}} + \Vert u_1\Vert_{H^s}\big].
    \end{equation}
    Combining these estimates, we get the estimate for the solution $u$. Now, to estimate $\Vert(-\Delta)^s u\Vert_{L^2}$, we need first to estimate $u_{tt}$. Reasoning as in \eqref{Estimate Delta u_t}, the first estimate for $u_t$ in \eqref{Energy3}, when applied to $u_t$ the solution to $\eqref{Equation in u_t}$ instead of $u$, gives
    \begin{align}\label{Estimate u_tt}
        \Vert u_{tt}(t,\cdot)\Vert_{L^2} & \lesssim \big(1+\Vert a\Vert_{L^{\frac{d}{2s}}}\big)^{\frac{1}{2}}\big[\Vert u_1\Vert_{H^s} + \Vert u_{tt}(0,\cdot)\Vert_{L^2}\big]\\
        & \nonumber \lesssim \big(1+\Vert a\Vert_{L^{\frac{d}{2s}}}\big)^{\frac{1}{2}}\big[\Vert u_1\Vert_{H^s} + \Vert u_0\Vert_{H^{2s}} + \Vert a\Vert_{L^{\frac{d}{s}}} \Vert u_0\Vert_{H^s} + \Vert b\Vert_{L^{\frac{d}{s}}} \Vert u_1\Vert_{H^s}\big]\\
        & \nonumber \lesssim \big(1+\Vert a\Vert_{L^{\frac{d}{2s}}}\big)\big(1+\Vert a\Vert_{L^{\frac{d}{s}}}\big)\big(1+\Vert b\Vert_{L^{\frac{d}{s}}}\big)\big[\Vert u_0\Vert_{H^{2s}} + \Vert u_1\Vert_{H^s}\big].
    \end{align}
    The estimate for $\Vert(-\Delta)^s u\Vert_{L^2}$ follows by taking the $L^2$ norm in the equality
    \begin{equation*}
        (-\Delta)^{s}u(t,x)=-u_{tt}(t,x) - a(x)u(t,x) - b(x)u_{t}(t,x),
    \end{equation*}
    and using the triangle inequality in the right hand side and by taking into consideration the so far obtained estimates \eqref{Estimate2 au}, \eqref{Estimate2 bu_t} and \eqref{Estimate u_tt}. This completes the proof.
\end{proof}

\section{Very weak well-posedness}
Here and in the sequel, we consider the case when the equation coefficients $a$, $b$ and the Cauchy data $u_0$ and $u_1$ are irregular (functions) and prove that the Cauchy problem
\begin{equation}\label{Equation sing}
    \bigg\lbrace
    \begin{array}{l}
    u_{tt}(t,x) + (-\Delta)^{s}u(t,x) + a(x)u(t,x) + b(x)u_{t}(t,x)=0, \,\, (t,x)\in [0,T]\times\mathbb{R}^d,\\
    u(0,x)=u_{0}(x),\quad u_{t}(0,x)=u_{1}(x), \quad x\in\mathbb{R}^d,
    \end{array}
\end{equation}
has a unique very weak solution. We have in mind ``functions" having $\delta$ or $\delta^2$-like behaviours. We note that we understand a multiplication of distributions as multiplication of approximating families, in particular the multiplication of their representatives in Colombeau algebra.

\subsection{Existence of very weak solutions}
In order to prove existence of very weak solutions to \eqref{Equation sing}, we need the following definitions.

\begin{defn}[\textbf{Friedrichs mollifier}]\label{Defn Friedrichs mollifier}
A function $\psi\in C_0^{\infty}(\mathbb{R}^d)$ is said to be a Friedrichs-mollifier if $\psi$ is non-negative and $\int_{\mathbb{R}^d}\psi(x)\d x=1$.
\end{defn}

\begin{exam}
An example of a Friedrichs-mollifier is given by:
\begin{equation*}
    \psi(x)=\left\lbrace
    \begin{array}{l}
    \alpha \e^{-\frac{1}{1-\vert x^2\vert}}\quad \vert x\vert <1,\\
    0\qquad\qquad \vert x\vert \geq 1,
    \end{array}
    \right.
\end{equation*}
where the constant $\alpha$ is choosed in such way that $\int_{\mathbb{R}^d}\psi(x)\d x=1$.
\end{exam}
Assume now $\psi$ as defined above a Friedrichs mollifier.

\begin{defn}[\textbf{Mollifying net}]\label{Defn Mollifying net}
For $\varepsilon\in(0,1]$, and $x\in\mathbb{R}^d$, a net of functions $\left(\psi_\varepsilon\right)_{\varepsilon\in(0,1]}$ is called a mollifying net if 
\begin{equation*}
    \psi_\varepsilon (x)=\omega(\varepsilon)^{-1}\psi\left(x/\omega(\varepsilon)\right),
\end{equation*}
where $\omega(\varepsilon)$ is a positive function converging to $0$ as $\varepsilon\rightarrow 0$ and $\psi$ is a Friedrichs-mollifier. In particular, if we take $\omega(\varepsilon)=\varepsilon$, then, we get
\begin{equation*}
    \psi_\varepsilon (x)=\varepsilon^{-1}\psi\left(x/\varepsilon\right).
\end{equation*}
\end{defn}

Given a function (distribution) $f$, regularising $f$ by convolution with a mollifying net $\left(\psi_\varepsilon\right)_{\varepsilon\in(0,1]}$, yields a net of smooth functions, namely
\begin{equation}
    (f_\varepsilon)_{\varepsilon\in(0,1]}=(f\ast \psi_\varepsilon)_{\varepsilon\in(0,1]}.\label{regularisation by convolution}
\end{equation}

\begin{rem}\label{Remark approx vs regularis}
    The term ``regularisation" of a function or distribution $f$ when used, will be viewed as a net of smooth functions $(f_{\varepsilon})_{\varepsilon\in (0,1]}$ arising from convolution with a mollifying net (as in Definition \ref{Defn Mollifying net}). However, the term ``approximation" is more general in the sense that approximations are not necessarily arising from convolution with molliffying nets. For instance, if we consider $(f_{\varepsilon})_{\varepsilon\in (0,1]}$ a regularisation of $f$, then the net of functions $(\Tilde{f}_{\varepsilon})_{\varepsilon\in (0,1]}$ defined by
    \begin{equation}\label{Approximation}         \Tilde{f}_{\varepsilon}=f_{\varepsilon}    + \e^{-\frac{1}{\varepsilon}},
    \end{equation}
    is an approximation of $f$ but not resulting from regularisation.
\end{rem}
Now, for a function (distribution) $f$, let $(f_\varepsilon)_{\varepsilon\in(0,1]}$ be a net of smooth functions approximating $f$, not necessarily coming from regularisation.
\begin{defn}[\textbf{Moderateness}]\label{Defn moderateness}
Let $X$ be a normed space of functions on $\mathbb{R}^d$ endowed with the norm $\Vert \cdot\Vert_X$.
\begin{enumerate}
    \item[\textsc{1.}] A net of functions $(f_\varepsilon)_{\varepsilon\in(0,1]}$ from $X$ is said to be $X$-moderate, if there exist $N\in\mathbb{N}_0$ such that
\begin{equation}\label{Defn moderateness1}
    \Vert f_\varepsilon\Vert_X \lesssim \omega(\varepsilon)^{-N}.
\end{equation}
    \item[\textsc{2.}] For $T>0$. A net of functions $(u_\varepsilon(\cdot,\cdot))_{\varepsilon\in(0,1]}$ from $C\big([0,T];H^{s}(\mathbb{R}^d)\big)\cap\newline C^1\big([0,T]; L^{2}(\mathbb{R}^d)\big)$ is said to be $C\big([0,T];H^{s}(\mathbb{R}^d)\big)\cap C^1\big([0,T];L^{2}(\mathbb{R}^d)\big)$-mode-rate, if there exist $N\in\mathbb{N}_0$ such that
    \begin{equation}\label{Defn moderateness2}
        \sup_{t\in [0,T]}\Vert u_\varepsilon(t,\cdot)\Vert_1 \lesssim \omega(\varepsilon)^{-N}.
    \end{equation}
    \item[\textsc{3.}] For $T>0$. A net of functions $(u_\varepsilon(\cdot,\cdot))_{\varepsilon\in(0,1]}$ from $C\big([0,T];H^{2s}(\mathbb{R}^d)\big)\cap C^1\big([0,T];H^{s}(\mathbb{R}^d)\big)$ is said to be $C\big([0,T];H^{2s}(\mathbb{R}^d)\big)\cap C^1\big([0,T];H^{s}(\mathbb{R}^d)\big)$-mode-rate, if there exist $N\in\mathbb{N}_0$ such that
    \begin{equation}\label{Defn moderateness3}
        \sup_{t\in [0,T]}\Vert u_\varepsilon(t,\cdot)\Vert_2 \lesssim \omega(\varepsilon)^{-N}.
    \end{equation}
\end{enumerate}
For the second and the third definitions of moderateness, we will shortly write $C_1$-moderate and $C_2$-moderate.
\end{defn}
The following proposition states that moderateness as defined above is a natural assumption for compactly supported distributions. Indeed, we have:
\begin{prop}
    Let $f\in \mathcal{E}^{'}(\mathbb{R}^d)$ and let $(f_\varepsilon)_{\varepsilon\in(0,1]}$ be regularisation of $f$ obtained via convolution with a mollifying net $\left(\psi_\varepsilon\right)_{\varepsilon\in(0,1]}$ (see Definition \ref{Defn Mollifying net}). Then, the net $(f_\varepsilon)_{\varepsilon\in(0,1]}$ is $L^{p}(\mathbb{R}^d)$-moderate for any $1\leq p\leq \infty$.
\end{prop}
\begin{proof}
    Fix $p\in [1,\infty]$ and let $f\in \mathcal{E}^{'}(\mathbb{R}^d)$. By the structure theorems for distributions (see \cite[Corollary 5.4.1]{FJ98}), there exists $n\in \mathbb{N}$ and compactly supported functions $f_{\alpha}\in C(\mathbb{R}^{d})$ such that
    \begin{equation*}
    T=\sum_{|\alpha| \leq n}\partial^{\alpha}f_{\alpha},
    \end{equation*}
    where $|\alpha|$ is the length of the multi-index $\alpha$. The convolution of $f$ with a mollifying net $\left(\psi_\varepsilon\right)_{\varepsilon\in(0,1]}$ yields
    \begin{equation}\label{Struct thm}
        f\ast\psi_{\varepsilon}=\sum_{\vert \alpha\vert \leq n}\partial^{\alpha}f_{\alpha}\ast\psi_{\varepsilon}=\sum_{\vert \alpha\vert \leq n}f_{\alpha}\ast\partial^{\alpha}\psi_{\varepsilon}=\sum_{\vert \alpha\vert \leq n}\varepsilon^{-d-\vert\alpha\vert}f_{\alpha}\ast\partial^{\alpha}\psi(x/\varepsilon).
    \end{equation}
    Taking the $L^p$ norm in \eqref{Struct thm} gives
    \begin{equation}\label{Struct thm1}
        \Vert f\ast\psi_{\varepsilon}\Vert_{L^p} \leq \sum_{\vert \alpha\vert \leq n}\varepsilon^{-d-\vert\alpha\vert}\Vert f_{\alpha}\ast\partial^{\alpha}\psi(x/\varepsilon)\Vert_{L^p}.
    \end{equation}
    Since $f_{\alpha}$ and $\psi$ are compactly supported then, Young's inequality applies for any $p_1,p_2 \in [1,\infty]$, provided that $\frac{1}{p_1}+\frac{1}{p_2}=\frac{1}{p}$. That is
    \begin{equation*}
        \Vert f_{\alpha}\ast\partial^{\alpha}\psi(x/\varepsilon)\Vert_{L^p} \leq \Vert f_{\alpha}\Vert_{L^{p_1}}\Vert\partial^{\alpha}\psi(x/\varepsilon)\Vert_{L^{p_2}} <\infty.
    \end{equation*}
    It follows from \eqref{Struct thm1} that $(f_\varepsilon)_{\varepsilon\in(0,1]}$ is $L^{p}(\mathbb{R}^d)$-moderate. 
\end{proof}

\begin{exam} \label{example moderateness} Let             $(\psi_\varepsilon)_\varepsilon$ be a mollifying      net such that $\psi_\varepsilon (x) =                 \varepsilon^{-1}\psi(\varepsilon^{-1}x)$. Since 
    $\psi$ is compactly supported, then,
    \begin{itemize}
    \item[(1)] For $f(x)=\delta_{0}(x)$, we have
    $f_{\varepsilon}(x) = \varepsilon^{-1}\psi(\varepsilon^{-1}x)\leq C\varepsilon^{-1}.$
    \item[(2)] For $f(x)=\delta_{0}^{2}(x)$, we can take
    $f_{\varepsilon}(x) = \varepsilon^{-2}\psi^{2}(\varepsilon^{-1}x) \leq C\varepsilon^{-2}.$
    \end{itemize}
\end{exam}

Now, we are ready to introduce the notion of very weak solutions adapted to our problem. Here and in the sequel, we consider $\omega(\varepsilon)=\varepsilon$, in all the above definitions.

\begin{defn}[\textbf{Very weak solution}]\label{Defn1 V.W.S}
    A net of functions $(u_{\varepsilon})_{\varepsilon}\in C([0,T];H^{s}(\mathbb{R}^d))\cap C^1([0,T];L^{2}(\mathbb{R}^d))$ is said to be a very weak solution to the Cauchy problem (\ref{Equation sing}), if there exist
    \begin{itemize}
        \item $L^{\infty}(\mathbb{R}^d)$-moderate approximations $(a_{\varepsilon})_{\varepsilon}$ and $(b_{\varepsilon})_{\varepsilon}$ to $a$ and $b$, with $a_{\varepsilon} \geq 0$ and $b_{\varepsilon} \geq 0$,
        \item $H^{s}(\mathbb{R}^d)$-moderate approximation $(u_{0,\varepsilon})_{\varepsilon}$ to $u_0$,
        \item $L^{2}(\mathbb{R}^d)$-moderate approximation $(u_{1,\varepsilon})_{\varepsilon}$ to $u_1$,
    \end{itemize}
    such that, $(u_{\varepsilon})_{\varepsilon}$ solves the approximating problems
    \begin{equation}\label{Approximating equation}
        \bigg\{
        \begin{array}{l}
        \partial_{t}^2u_{\varepsilon}(t,x) + (-\Delta)^{s}u_{\varepsilon}(t,x) + a_{\varepsilon}(x)u_{\varepsilon}(t,x) + b_{\varepsilon}(x)\partial_{t}u_{\varepsilon}(t,x)=0, \,\, (t,x)\in [0,T]\times\mathbb{R}^d,\\
        u_{\varepsilon}(0,x)=u_{0,\varepsilon}(x),\quad \partial_{t}u_{\varepsilon}(0,x)=u_{1,\varepsilon}(x), \quad x\in\mathbb{R}^d,
        \end{array}
    \end{equation}
    for all $\varepsilon\in (0,1]$, and is $C_1$-moderate.
\end{defn}

We have also the following alternative definition of a very weak solution to \eqref{Equation sing}, under the assumptions of Lemma \ref{Lemma2}.

\begin{defn}\label{Defn2 V.W.S}
    Let $d>2s$. A net of functions $(u_{\varepsilon})_{\varepsilon}\in C([0,T];H^{2s}(\mathbb{R}^d))\cap C^1([0,T];H^{s}(\mathbb{R}^d))$ is said to be a very weak solution to the Cauchy problem (\ref{Equation sing}), if there exist
    \begin{itemize}
        \item $(L^{\frac{d}{s}}(\mathbb{R}^d)\cap L^{\frac{d}{2s}}(\mathbb{R}^d))$-moderate approximation $(a_{\varepsilon})_{\varepsilon}$ to $a$, with $a_{\varepsilon} \geq 0$,
        \item $L^{\frac{d}{s}}(\mathbb{R}^d)$-moderate approximation $(b_{\varepsilon})_{\varepsilon}$ to $b$, with $b_{\varepsilon} \geq 0$,
        \item $H^{2s}(\mathbb{R}^d)$-moderate approximation $(u_{0,\varepsilon})_{\varepsilon}$ to $u_0$,
        \item $H^{s}(\mathbb{R}^d)$-moderate approximation $(u_{1,\varepsilon})_{\varepsilon}$ to $u_1$,
    \end{itemize}
    such that, $(u_{\varepsilon})_{\varepsilon}$ solves the approximating problems (as in Definition \ref{Defn1 V.W.S}) for all $\varepsilon\in (0,1]$, and is $C_2$-moderate.
\end{defn}
Now, under the assumptions in Definition \ref{Defn1 V.W.S} and Definition \ref{Defn2 V.W.S}, the existence of a very weak solution is straightforward.

\begin{thm}\label{Thm1 existence}
    Assume that there exist $\big\{L^{\infty}(\mathbb{R}^d),L^{\infty}(\mathbb{R}^d),H^{s}(\mathbb{R}^d),L^2(\mathbb{R}^d)\big\}$-moderate approximations to $a,b,u_0$ and $u_1$ respectively, with $a_{\varepsilon} \geq 0$ and $b_{\varepsilon} \geq 0$. Then, the Cauchy problem \eqref{Equation sing} has a very weak solution.
\end{thm}

\begin{proof}
    Let $a,b,u_0$ and $u_1$ as in assumptions. Then, there exists $N_1,N_2,N_3,N_4 \in \mathbb{N}$, such that
    \begin{equation*}
        \Vert a_{\varepsilon}\Vert_{L^{\infty}} \lesssim \varepsilon^{-N_1},\quad \Vert b_{\varepsilon}\Vert_{L^{\infty}} \lesssim \varepsilon^{-N_2},
    \end{equation*}
    and 
    \begin{equation*}
        \Vert u_{0,\varepsilon}\Vert_{H^s} \lesssim \varepsilon^{-N_3},\quad \Vert u_{1,\varepsilon}\Vert_{H^s} \lesssim \varepsilon^{-N_4}.
    \end{equation*}
    It follows from the energy estimate \eqref{Energy estimate}, that
    \begin{equation*}
        \Vert u_{\varepsilon}(t,\cdot)\Vert_1 \lesssim \varepsilon^{-N_1-N_2-\max \{N_3,N_4\}},
    \end{equation*}
    uniformly in $t\in [0,T]$, which means that the net $(u_{\varepsilon})_{\varepsilon}$ is $C_1$-moderate. This concludes the proof.
\end{proof}

As an alternative to Theorem \ref{Thm1 existence} in the case when $d>2s$ and the equation coefficients and data satisfy the hypothesis of Definition \ref{Defn2 V.W.S}, we have the following theorem for which we do not give the proof, since it is similar to the one of Theorem \ref{Thm1 existence}.

\begin{thm}\label{Thm2 existence}
    Assume that there exist $\big\{(L^{\frac{d}{s}}(\mathbb{R}^d)\cap L^{\frac{d}{2s}}(\mathbb{R}^d)),L^{\frac{d}{s}}(\mathbb{R}^d),H^{2s}(\mathbb{R}^d),H^{s}(\mathbb{R}^d)\big\}$-moderate approximations to $a,b,u_0$ and $u_1$ respectively, with $a_{\varepsilon} \geq 0$ and $b_{\varepsilon} \geq 0$. Then, the Cauchy problem \eqref{Equation sing} has a very weak solution.
\end{thm}

\subsection{Uniqueness}
In what follows we want to prove the uniqueness of the very weak solution to the Cauchy problem \eqref{Equation sing} in both situations, either in the case when very weak solutions exist with the assumptions of Theorem \ref{Thm1 existence} or in the case of Theorem \ref{Thm2 existence}. We need the following definition.

\begin{defn}[\textbf{Negligibility}]\label{Defn negligibility}
    Let $X$ be a normed space endowed with the norm $\Vert\cdot\Vert_X$. A net of functions $(f_\varepsilon)_{\varepsilon\in(0,1]}$ from $X$ is said to be $X$-negligible, if the estimate
    \begin{equation}\label{Negligibility formula}
        \Vert f_\varepsilon\Vert_X \lesssim \varepsilon^k,
    \end{equation}
    is valid for all $k>0$.
\end{defn}
Roughly speaking, we understand the uniqueness of the very weak solution to the Cauchy problem \eqref{Equation sing}, in the sense that negligible changes in the approximations of the equation coefficients and initial data, lead to negligible changes in the corresponding very weak solutions. More precisely,
\begin{defn}[\textbf{Uniqueness}]\label{Defn1 uniqueness}
    We say that the Cauchy problem \eqref{Equation sing}, has a unique very weak solution, if for all families of approximations $(a_{\varepsilon})_{\varepsilon}$, $(\Tilde{a}_{\varepsilon})_{\varepsilon}$ and $(b_{\varepsilon})_{\varepsilon}$, $(\Tilde{b}_{\varepsilon})_{\varepsilon}$ for the equation coefficients $a$ and $b$, and families of approximations $(u_{0,\varepsilon})_{\varepsilon}$, $(\Tilde{u}_{0,\varepsilon})_{\varepsilon}$ and  $(u_{1,\varepsilon})_{\varepsilon}$, $(\Tilde{u}_{1,\varepsilon})_{\varepsilon}$ for the Cauchy data $u_0$ and $u_1$, such that the nets $(a_{\varepsilon}-\Tilde{a}_{\varepsilon})_{\varepsilon}$, $(b_{\varepsilon}-\Tilde{b}_{\varepsilon})_{\varepsilon}$, $(u_{0,\varepsilon}-\Tilde{u}_{0,\varepsilon})_{\varepsilon}$ and $(u_{1,\varepsilon}-\Tilde{u}_{1,\varepsilon})_{\varepsilon}$ are $\big\{L^{\infty}(\mathbb{R}^d),L^{\infty}(\mathbb{R}^d),H^{s}(\mathbb{R}^d),L^2(\mathbb{R}^d)\big\}$-negligible, it follows that the net
    \begin{equation*}
        \big(u_{\varepsilon}(t,\cdot)-\Tilde{u}_{\varepsilon}(t,\cdot)\big)_{\varepsilon}
    \end{equation*}
    is $L^2(\mathbb{R}^d)$-negligible for all $t\in [0,t]$, where $(u_{\varepsilon})_{\varepsilon}$ and $(\Tilde{u}_{\varepsilon})_{\varepsilon}$ are the families of solutions to the approximating Cauchy problems
    \begin{equation}\label{Approximating equation for u}
        \bigg\{
        \begin{array}{l}
        \partial_{t}^2u_{\varepsilon}(t,x) + (-\Delta)^{s}u_{\varepsilon}(t,x) + a_{\varepsilon}(x)u_{\varepsilon}(t,x) + b_{\varepsilon}(x)\partial_{t}u_{\varepsilon}(t,x)=0, \,\, (t,x)\in [0,T]\times\mathbb{R}^d,\\
        u_{\varepsilon}(0,x)=u_{0,\varepsilon}(x),\quad \partial_{t}u_{\varepsilon}(0,x)=u_{1,\varepsilon}(x), \quad x\in\mathbb{R}^d,
        \end{array}
    \end{equation}
    and
    \begin{equation}\label{Approximating equation for u tilde}
        \bigg\{
        \begin{array}{l}
        \partial_{t}^2\Tilde{u}_{\varepsilon}(t,x) + (-\Delta)^{s}\Tilde{u}_{\varepsilon}(t,x) + \Tilde{a}_{\varepsilon}(x)\Tilde{u}_{\varepsilon}(t,x) + \Tilde{b}_{\varepsilon}(x)\partial_{t}\Tilde{u}_{\varepsilon}(t,x)=0, \,\, (t,x)\in [0,T]\times\mathbb{R}^d,\\
        \Tilde{u}_{\varepsilon}(0,x)=\Tilde{u}_{0,\varepsilon}(x),\quad \partial_{t}\Tilde{u}_{\varepsilon}(0,x)=\Tilde{u}_{1,\varepsilon}(x), \quad x\in\mathbb{R}^d,
        \end{array}
    \end{equation}
    respectively.
\end{defn}

\begin{thm}\label{Thm1 uniqueness}
    Assume that $a,b \geq 0$, in the sense that their approximating nets are non-negative. Under the conditions of Theorem \ref{Thm1 existence}, the very weak solution to the Cauchy problem \eqref{Equation sing} is unique.
\end{thm}

\begin{proof}
    Let $(u_{\varepsilon})_{\varepsilon}$ and $(\Tilde{u}_{\varepsilon})_{\varepsilon}$ be the families of solutions to \eqref{Approximating equation for u} and \eqref{Approximating equation for u tilde} and assume that the nets $(a_{\varepsilon}-\Tilde{a}_{\varepsilon})_{\varepsilon}$, $(b_{\varepsilon}-\Tilde{b}_{\varepsilon})_{\varepsilon}$, $(u_{0,\varepsilon}-\Tilde{u}_{0,\varepsilon})_{\varepsilon}$ and $(u_{1,\varepsilon}-\Tilde{u}_{1,\varepsilon})_{\varepsilon}$ are $L^{\infty}(\mathbb{R}^d)$, $L^{\infty}(\mathbb{R}^d)$, $H^{s}(\mathbb{R}^d)$, $L^2(\mathbb{R}^d)$-negligible, respectively. The function $U_{\varepsilon}(t,x)$ defined by
    \begin{equation*}
        U_{\varepsilon}(t,x):=u_{\varepsilon}(t,x)-\Tilde{u}_{\varepsilon}(t,x),
    \end{equation*}
    satisfies
    \begin{equation}\label{Equation1 U}
        \bigg\{
        \begin{array}{l}
        \partial_{t}^2U_{\varepsilon}(t,x) + (-\Delta)^{s}U_{\varepsilon}(t,x) + a_{\varepsilon}(x)U_{\varepsilon}(t,x) + b_{\varepsilon}(x)\partial_{t}U_{\varepsilon}(t,x)=f_{\varepsilon}(t,x),\\
        U_{\varepsilon}(0,x)=(u_{0,\varepsilon}-\Tilde{u}_{0,\varepsilon})(x),\quad \partial_{t}U_{\varepsilon}(0,x)=(u_{1,\varepsilon}-\Tilde{u}_{1,\varepsilon})(x),
        \end{array}
    \end{equation}
    for $(t,x)\in [0,T]\times\mathbb{R}^d$, where,
    \begin{equation*}
        f_{\varepsilon}(t,x):=\big(\Tilde{a}_{\varepsilon}(x)-a_{\varepsilon}(x)\big)\Tilde{u}_{\varepsilon}(t,x) + \big(\Tilde{b}_{\varepsilon}(x)-b_{\varepsilon}(x)\big)\partial_{t}\Tilde{u}_{\varepsilon}(t,x).
    \end{equation*}
    According to Duhamel's principle (see Proposition \ref{Prop Duhamel}), the solution to \eqref{Equation1 U} has the following representation
    \begin{equation}\label{Duhamel1 solution}
        U_{\varepsilon}(t,x) = W_{\varepsilon}(t,x) + \int_{0}^{t}V_{\varepsilon}(t,x;\tau)\d \tau,
    \end{equation}
    where $W_{\varepsilon}(t,x)$ is the solution to the homogeneous problem
    \begin{equation}\label{Equation1 W}
        \bigg\{
        \begin{array}{l}
        \partial_{t}^2W_{\varepsilon}(t,x) + (-\Delta)^{s}W_{\varepsilon}(t,x) + a_{\varepsilon}(x)W_{\varepsilon}(t,x) + b_{\varepsilon}(x)\partial_{t}W_{\varepsilon}(t,x)=0,\\
        W_{\varepsilon}(0,x)=(u_{0,\varepsilon}-\Tilde{u}_{0,\varepsilon})(x),\quad \partial_{t}W_{\varepsilon}(0,x)=(u_{1,\varepsilon}-\Tilde{u}_{1,\varepsilon})(x),
        \end{array}
    \end{equation}
    for $(t,x)\in [0,T]\times\mathbb{R}^d$, and $V_{\varepsilon}(t,x;\tau)$ solves
    \begin{equation}\label{Equation1 V}
        \bigg\{
        \begin{array}{l}
        \partial_{t}^2V_{\varepsilon}(t,x;\tau) + (-\Delta)^{s}V_{\varepsilon}(t,x;\tau) + a_{\varepsilon}(x)V_{\varepsilon}(t,x;\tau) + b_{\varepsilon}(x)\partial_{t}V_{\varepsilon}(t,x;\tau)=0,\\
        V_{\varepsilon}(\tau,x;\tau)=0,\quad \partial_{t}V_{\varepsilon}(\tau,x;\tau)=f_{\varepsilon}(\tau,x),
        \end{array}
    \end{equation}
    for $(t,x)\in [\tau,T]\times\mathbb{R}^d$ and $\tau\in [0,T]$.
    By taking the $L^2$-norm on both sides of \eqref{Duhamel1 solution} and using Minkowski's integral inequality, we get
    \begin{equation}\label{Duhamel1 solution estimate}
        \Vert U_{\varepsilon}(t,\cdot)\Vert_{L^2} \leq \Vert W_{\varepsilon}(t,\cdot)\Vert_{L^2} + \int_{0}^{t}\Vert V_{\varepsilon}(t,\cdot;\tau)\Vert_{L^2}\d \tau.
    \end{equation}
    The energy estimate \eqref{Energy estimate} allows us to control $\Vert W_{\varepsilon}(t,\cdot)\Vert_{L^2}$ and $\Vert V_{\varepsilon}(t,\cdot;\tau)\Vert_{L^2}$ to get
    \begin{equation*}
        \Vert W_{\varepsilon}(t,\cdot)\Vert_{L^2} \lesssim \Big(1+\Vert a_{\varepsilon}\Vert_{L^{\infty}}\Big)\Big(1+\Vert b_{\varepsilon}\Vert_{L^{\infty}}\Big)\bigg[\Vert u_{0,\varepsilon}-\Tilde{u}_{0,\varepsilon}\Vert_{H^s} + \Vert u_{1,\varepsilon}-\Tilde{u}_{1,\varepsilon}\Vert_{L^2}\bigg],
    \end{equation*}
    and
    \begin{equation*}
        \Vert V_{\varepsilon}(t,\cdot;\tau)\Vert_{L^2} \lesssim \Big(1+\Vert a_{\varepsilon}\Vert_{L^{\infty}}\Big)\Big(1+\Vert b_{\varepsilon}\Vert_{L^{\infty}}\Big)\bigg[\Vert f_{\varepsilon}(\tau,\cdot)\Vert_{L^2}\bigg].
    \end{equation*}
    By taking into consideration that $t\in[0,T]$, it follows from \eqref{Duhamel1 solution estimate} that
    \begin{align}\label{Estimate U}
        \Vert U_{\varepsilon}(t,\cdot)\Vert_{L^2} \lesssim \Big(1+\Vert a_{\varepsilon}\Vert_{L^{\infty}}\Big)\Big(1+\Vert b_{\varepsilon}\Vert_{L^{\infty}}\Big)& \bigg[\Vert u_{0,\varepsilon}-\Tilde{u}_{0,\varepsilon}\Vert_{H^s} +\\
        & \nonumber\Vert u_{1,\varepsilon}-\Tilde{u}_{1,\varepsilon}\Vert_{L^2} + \int_{0}^{T}\Vert f_{\varepsilon}(\tau,\cdot)\Vert_{L^2}\d \tau\bigg],
    \end{align}
    where $\Vert f_{\varepsilon}(\tau,\cdot)\Vert_{L^2}$ is estimated as follows,
    \begin{align}\label{Estimate f_epsilon}
        \Vert f_{\varepsilon}(\tau,\cdot)\Vert_{L^2} & \leq \Vert (\Tilde{a}_{\varepsilon}(\cdot)-a_{\varepsilon}(\cdot))\Tilde{u}_{\varepsilon}(\tau,\cdot)\Vert_{L^2} + \Vert(\Tilde{b}_{\varepsilon}(\cdot)-b_{\varepsilon}(\cdot))\partial_{t}\Tilde{u}_{\varepsilon}(\tau,\cdot)\Vert_{L^2}\\
        & \nonumber\leq \Vert \Tilde{a}_{\varepsilon} - a_{\varepsilon}\Vert_{L^{\infty}}\Vert \Tilde{u}_{\varepsilon}(\tau,\cdot)\Vert_{L^2} + \Vert \Tilde{b}_{\varepsilon} - b_{\varepsilon}\Vert_{L^{\infty}}\Vert \partial_{t}\Tilde{u}_{\varepsilon}(\tau,\cdot)\Vert_{L^2}.
    \end{align}
    On the one hand, the nets $(a_{\varepsilon})_{\varepsilon}$ and $(b_{\varepsilon})_{\varepsilon}$ are $L^{\infty}$-moderate by assumption, and the net $(\Tilde{u}_{\varepsilon})_{\varepsilon}$ is $C_1$-moderate being a very weak solution to \eqref{Approximating equation for u tilde}. On the other hand, the nets $(a_{\varepsilon}-\Tilde{a}_{\varepsilon})_{\varepsilon}$, $(b_{\varepsilon}-\Tilde{b}_{\varepsilon})_{\varepsilon}$, $(u_{0,\varepsilon}-\Tilde{u}_{0,\varepsilon})_{\varepsilon}$ and $(u_{1,\varepsilon}-\Tilde{u}_{1,\varepsilon})_{\varepsilon}$ are $L^{\infty}(\mathbb{R}^d)$, $L^{\infty}(\mathbb{R}^d)$, $H^{s}(\mathbb{R}^d)$, $L^2(\mathbb{R}^d)$-negligible. It follows from \eqref{Estimate U} combined with \eqref{Estimate f_epsilon} that
    \begin{equation*}
        \Vert U_{\varepsilon}(t,\cdot)\Vert_{L^2} \lesssim \varepsilon^{k},
    \end{equation*}
    for all $k>0$, showing the uniqueness of the very weak solution.    
\end{proof}
The analogue to Definition \ref{Defn1 uniqueness} and Theorem \ref{Thm1 uniqueness} in the case when $d>2s$ with Theorem \ref{Thm2 existence}'s background, read:

\begin{defn}\label{Defn2 uniqueness}
    We say that the Cauchy problem \eqref{Equation sing}, has a unique very weak solution, if for all families of approximations $(a_{\varepsilon})_{\varepsilon}$, $(\Tilde{a}_{\varepsilon})_{\varepsilon}$ and $(b_{\varepsilon})_{\varepsilon}$, $(\Tilde{b}_{\varepsilon})_{\varepsilon}$ for the equation coefficients $a$ and $b$, and families of approximations $(u_{0,\varepsilon})_{\varepsilon}$, $(\Tilde{u}_{0,\varepsilon})_{\varepsilon}$ and  $(u_{1,\varepsilon})_{\varepsilon}$, $(\Tilde{u}_{1,\varepsilon})_{\varepsilon}$ for the Cauchy data $u_0$ and $u_1$, such that the nets $(a_{\varepsilon}-\Tilde{a}_{\varepsilon})_{\varepsilon}$, $(b_{\varepsilon}-\Tilde{b}_{\varepsilon})_{\varepsilon}$, $(u_{0,\varepsilon}-\Tilde{u}_{0,\varepsilon})_{\varepsilon}$ and $(u_{1,\varepsilon}-\Tilde{u}_{1,\varepsilon})_{\varepsilon}$ are $\big\{(L^{\frac{d}{s}}(\mathbb{R}^d)\cap L^{\frac{d}{2s}}(\mathbb{R}^d)),L^{\frac{d}{s}}(\mathbb{R}^d),H^{2s}(\mathbb{R}^d),H^{s}(\mathbb{R}^d)\big\}$-negligible, it follows that the net $\big(u_{\varepsilon}(t,\cdot)-\Tilde{u}_{\varepsilon}(t,\cdot)\big)_{\varepsilon\in (0,1]}$, is $L^2(\mathbb{R}^d)$-negligible for all $t\in [0,T]$, where $(u_{\varepsilon})_{\varepsilon}$ and $(\Tilde{u}_{\varepsilon})_{\varepsilon}$ are the families of solutions to the corresponding approximating Cauchy problems.
\end{defn}

\begin{thm}\label{Thm2 uniqueness}
    Let $d>2s$ and assume that $a,b \geq 0$, in the sense that there approximating nets are non-negative. With the assumptions of Theorem \ref{Thm2 existence}, the very weak solution to the Cauchy problem \eqref{Equation sing} is unique.
\end{thm}

\section{Coherence with classical theory}
The question to be answered here is that, in the case when $a,b\in L^{\infty}(\mathbb{R}^d)$, $u_0 \in H^{s}(\mathbb{R}^d)$ and $u_1 \in L^{2}(\mathbb{R}^d)$ or alternatively when $(a,b)\in (L^{\frac{d}{s}}(\mathbb{R}^d)\cap L^{\frac{d}{2s}}(\mathbb{R}^d))\times L^{\frac{d}{s}}(\mathbb{R}^d)$, $u_0 \in H^{2s}(\mathbb{R}^d)$ and $u_1 \in H^{s}(\mathbb{R}^d)$ and a classical solution to the Cauchy problem
\begin{equation}\label{Equation coherence}
        \bigg\lbrace
        \begin{array}{l}
        u_{tt}(t,x) + (-\Delta)^{s}u(t,x) + a(x)u(t,x) + b(x)u_{t}(t,x)=0, \,\, (t,x)\in [0,T]\times\mathbb{R}^d,\\
        u(0,x)=u_{0}(x),\quad u_{t}(0,x)=u_{1}(x), \quad x\in\mathbb{R}^d,
        \end{array}
\end{equation}
exists, does the very weak solution obtained via regularisation techniques recapture it?

\begin{thm}\label{Thm1 coherence}
    Let $\psi$ be a Friedrichs-mollifier. Assume $a,b\in L^{\infty}(\mathbb{R}^d)$ be non-negative and suppose that $u_0 \in H^{s}(\mathbb{R}^d)$ and $u_1 \in L^{2}(\mathbb{R}^d)$. Then, for any regularising families $(a_{\varepsilon})_{\varepsilon}=(a\ast\psi_{\varepsilon})_{\varepsilon}$ and $(b_{\varepsilon})_{\varepsilon}=(b\ast\psi_{\varepsilon})_{\varepsilon}$ for the equation coefficients, satisfying
    \begin{equation}\label{approx.condition1}
        \lVert a_{\varepsilon} - a\rVert_{L^{\infty}} \rightarrow 0,\quad\text{and}\quad \lVert b_{\varepsilon} - b\rVert_{L^{\infty}} \rightarrow 0,
    \end{equation} 
    and any regularising families $(u_{0,\varepsilon})_{\varepsilon}=(u_{0}\ast\psi_{\varepsilon})_{\varepsilon}$ and $(u_{1,\varepsilon})_{\varepsilon}=(u_{1}\ast\psi_{\varepsilon})_{\varepsilon}$ for the initial data, the net $(u_{\varepsilon})_{\varepsilon}$ converges to the classical solution (given by Lemma \ref{Lemma1}) of the Cauchy problem (\ref{Equation coherence}) in $L^{2}$ as $\varepsilon \rightarrow 0$.
\end{thm}

\begin{proof}
    Let $(u_{\varepsilon})_{\varepsilon}$ be the very weak solution given by Theorem \ref{Thm1 existence} and $u$ the classical one, as in Lemma \ref{Lemma1}. The classical solution satisfies
    \begin{equation}\label{Equation coherence classic}
        \bigg\lbrace
        \begin{array}{l}
        u_{tt}(t,x) + (-\Delta)^{s}u(t,x) + a(x)u(t,x) + b(x)u_{t}(t,x)=0, \,\, (t,x)\in [0,T]\times\mathbb{R}^d,\\
        u(0,x)=u_{0}(x),\quad u_{t}(0,x)=u_{1}(x), \quad x\in\mathbb{R}^d,
        \end{array}
    \end{equation}
    and $(u_{\varepsilon})_{\varepsilon}$ solves
    \begin{equation}\label{Equation coherence VWS}
        \bigg\{
        \begin{array}{l}
        \partial_{t}^2u_{\varepsilon}(t,x) + (-\Delta)^{s}u_{\varepsilon}(t,x) + a_{\varepsilon}(x)u_{\varepsilon}(t,x) + b_{\varepsilon}(x)\partial_{t}u_{\varepsilon}(t,x)=0, \,\, (t,x)\in [0,T]\times\mathbb{R}^d,\\
        u_{\varepsilon}(0,x)=u_{0,\varepsilon}(x),\quad \partial_{t}u_{\varepsilon}(0,x)=u_{1,\varepsilon}(x), \quad x\in\mathbb{R}^d.
        \end{array}
    \end{equation}
    Denoting $U_{\varepsilon}(t,x):=u_{\varepsilon}(t,x)-u(t,x)$, we have that $U_{\varepsilon}$ solves the Cauchy problem
    \begin{equation}\label{Equation coherence V}
        \bigg\{
        \begin{array}{l}
        \partial_{t}^2U_{\varepsilon}(t,x) + (-\Delta)^{s}U_{\varepsilon}(t,x) + a_{\varepsilon}(x)U_{\varepsilon}(t,x) + b_{\varepsilon}(x)\partial_{t}U_{\varepsilon}(t,x)=\Theta_{\varepsilon}(t,x),\\
        U_{\varepsilon}(0,x)=(u_{0,\varepsilon}-u_0)(x),\quad \partial_{t}U_{\varepsilon}(0,x)=(u_{1,\varepsilon}-u_1)(x),
        \end{array}
    \end{equation}
    where
    \begin{equation}
        \Theta_{\varepsilon}(t,x):= -\big(a_{\varepsilon}(x)-a(x)\big)u(t,x) - \big(b_{\varepsilon}(x)-b(x)\big)\partial_{t}u(t,x).
    \end{equation}
    Thanks to Duhamel's principle, $U_{\varepsilon}$ can be represented by
    \begin{equation}\label{Duhamel2 solution}
        U_{\varepsilon}(t,x) = W_{\varepsilon}(t,x) + \int_{0}^{t}V_{\varepsilon}(t,x;\tau)\d \tau,
    \end{equation}
    where $W_{\varepsilon}(t,x)$ is the solution to the homogeneous problem
    \begin{equation}\label{Equation2 W}
        \bigg\{
        \begin{array}{l}
        \partial_{t}^2W_{\varepsilon}(t,x) + (-\Delta)^{s}W_{\varepsilon}(t,x) + a_{\varepsilon}(x)W_{\varepsilon}(t,x) + b_{\varepsilon}(x)\partial_{t}W_{\varepsilon}(t,x)=0,\\
        W_{\varepsilon}(0,x)=(u_{0,\varepsilon}-u_0)(x),\quad \partial_{t}W_{\varepsilon}(0,x)=(u_{1,\varepsilon}-u_1)(x),
        \end{array}
    \end{equation}
    for $(t,x)\in [0,T]\times\mathbb{R}^d$, and $V_{\varepsilon}(t,x;\tau)$ solves
    \begin{equation}\label{Equation2 V}
        \bigg\{
        \begin{array}{l}
        \partial_{t}^2V_{\varepsilon}(t,x;\tau) + (-\Delta)^{s}V_{\varepsilon}(t,x;\tau) + a_{\varepsilon}(x)V_{\varepsilon}(t,x;\tau) + b_{\varepsilon}(x)\partial_{t}V_{\varepsilon}(t,x;\tau)=0,\\
        V_{\varepsilon}(\tau,x;\tau)=0,\quad \partial_{t}V_{\varepsilon}(\tau,x;\tau)=\Theta_{\varepsilon}(\tau,x),
        \end{array}
    \end{equation}
    for $(t,x)\in [\tau,T]\times\mathbb{R}^d$ and $\tau\in [0,T]$. We take the $L^2$-norm in \eqref{Duhamel2 solution} and we argue as in the proof of Theorem \ref{Thm1 uniqueness}. We obtain
    \begin{equation}\label{Duhamel2 solution estimate}
        \Vert U_{\varepsilon}(t,\cdot)\Vert_{L^2} \leq \Vert W_{\varepsilon}(t,\cdot)\Vert_{L^2} + \int_{0}^{t}\Vert V_{\varepsilon}(t,\cdot;\tau)\Vert_{L^2}\d \tau,
    \end{equation}
    where
    \begin{equation*}
        \Vert W_{\varepsilon}(t,\cdot)\Vert_{L^2} \lesssim \Big(1+\Vert a_{\varepsilon}\Vert_{L^{\infty}}\Big)\Big(1+\Vert b_{\varepsilon}\Vert_{L^{\infty}}\Big)\bigg[\Vert u_{0,\varepsilon}-u_{0}\Vert_{H^s} + \Vert u_{1,\varepsilon}-u_{1,}\Vert_{L^2}\bigg],
    \end{equation*}
    and
    \begin{equation*}
        \Vert V_{\varepsilon}(t,\cdot;\tau)\Vert_{L^2} \lesssim \Big(1+\Vert a_{\varepsilon}\Vert_{L^{\infty}}\Big)\Big(1+\Vert b_{\varepsilon}\Vert_{L^{\infty}}\Big)\bigg[\Vert \Theta_{\varepsilon}(\tau,\cdot)\Vert_{L^2}\bigg],
    \end{equation*}
    by the energy estimate from Lemma \ref{Lemma1}, and $\Theta_{\varepsilon}$ is estimated by
    \begin{equation}\label{Estimate Theta_epsilon}
        \Vert \Theta_{\varepsilon}(\tau,\cdot)\Vert_{L^2} \leq \Vert a_{\varepsilon} - a\Vert_{L^{\infty}}\Vert u(\tau,\cdot)\Vert_{L^2} + \Vert b_{\varepsilon} - b\Vert_{L^{\infty}}\Vert \partial_{t}u(\tau,\cdot)\Vert_{L^2}.
    \end{equation}
    First, one observes that $\Vert a_{\varepsilon}\Vert_{L^{\infty}}<\infty$ and $\Vert b_{\varepsilon}\Vert_{L^{\infty}}<\infty$ uniformly in $\varepsilon$ by the fact that $a,b\in L^{\infty}(\mathbb{R}^d)$ and $\Vert u(\tau,\cdot)\Vert_{L^2}$ and $\Vert \partial_{t}u(\tau,\cdot)\Vert_{L^2}$ are bounded as well, since $u$ is a classical solution to \eqref{Equation coherence}. This, together with
    \begin{equation*}
        \lVert a_{\varepsilon} - a\rVert_{L^{\infty}} \rightarrow 0,\quad\text{and}\quad \lVert b_{\varepsilon} - b\rVert_{L^{\infty}} \rightarrow 0,\quad \text{as } \varepsilon\rightarrow 0,
    \end{equation*}
    from the assumptions, and
    \begin{equation*}
        \Vert u_{0,\varepsilon}-u_{0}\Vert_{H^s} \rightarrow 0,\quad  \Vert u_{1,\varepsilon}-u_{1,}\Vert_{L^2} \rightarrow 0,\quad \text{as } \varepsilon\rightarrow 0,
    \end{equation*}
    shows that
    \begin{equation*}
        \Vert U_{\varepsilon}(t,\cdot)\Vert_{L^2} \rightarrow 0,\quad \text{as } \varepsilon\rightarrow 0,
    \end{equation*}
    uniformly in $t\in [0,T]$, and this finishes the proof.    
\end{proof}

In the case when a classical solution exists in the sense of Lemma \ref{Lemma2}, the coherence theorem reads as follows. We avoid giving the proof since it is similar to the proof of Theorem \ref{Thm1 coherence}.

\begin{thm}\label{Thm2 coherence}
    Let $\psi$ be a Friedrichs-mollifier. Assume $(a,b)\in (L^{\frac{d}{s}}(\mathbb{R}^d)\cap L^{\frac{d}{2s}}(\mathbb{R}^d))\times L^{\frac{d}{s}}(\mathbb{R}^d)$ be non-negative and suppose that $u_0 \in H^{2s}(\mathbb{R}^d)$ and $u_1 \in H^{s}(\mathbb{R}^d)$. Then, for any regularising families $(a_{\varepsilon})_{\varepsilon}=(a\ast\psi_{\varepsilon})_{\varepsilon}$ and $(b_{\varepsilon})_{\varepsilon}=(b\ast\psi_{\varepsilon})_{\varepsilon}$ for the equation coefficients, and any regularising families $(u_{0,\varepsilon})_{\varepsilon}=(u_{0}\ast\psi_{\varepsilon})_{\varepsilon}$ and $(u_{1,\varepsilon})_{\varepsilon}=(u_{1}\ast\psi_{\varepsilon})_{\varepsilon}$ for the initial data, the net $(u_{\varepsilon})_{\varepsilon}$ converges to the classical solution (given by Lemma \ref{Lemma2}) of the Cauchy problem \eqref{Equation coherence} in $L^{2}$ as $\varepsilon \rightarrow 0$.
\end{thm}

\begin{rem}
    In Theorem \ref{Thm1 coherence}, we proved the coherence result, provided that
    \begin{equation*}
        \lVert a_{\varepsilon} - a\rVert_{L^{\infty}} \rightarrow 0,\quad\text{and}\quad \lVert b_{\varepsilon} - b\rVert_{L^{\infty}} \rightarrow 0,
    \end{equation*}
    as $\varepsilon\rightarrow 0$. This is in particular true if we consider coefficients from $C_0 (\mathbb{R}^d)$, the space of continuous functions on $\mathbb{R}^d$ vanishing at infinity which is a Banach space when endowed with the $L^{\infty}$-norm. For more details, see Section 3.1.10 in \cite{FR16}.
\end{rem}

\end{document}